\documentclass[twoside,a4paper]{article}

\usepackage{amsmath}
\usepackage{amsfonts}
\usepackage{amsthm}
\usepackage{enumitem}
\usepackage{url}

\theoremstyle{definition}
\newtheorem{definition}{Definition}[section]
\newtheorem{example}[definition]{Example}

\theoremstyle{plain}
\newtheorem{lemma}[definition]{Lemma}
\newtheorem{theorem}[definition]{Theorem}
\newtheorem{proposition}[definition]{Proposition}
\newtheorem{corollary}[definition]{Corollary}

\theoremstyle{remark}
\newtheorem{remark}[definition]{Remark}

\newcommand{\R}{\mathbb{R}}                             
\newcommand{\inner}[2]{\left\langle #1,#2\right\rangle} 

\DeclareMathOperator{\LspaceSymbol}{\mathbf{L}} 
\newcommand{\Lpspace}[1][p]{\LspaceSymbol^{{#1}}}
\newcommand{\LGspace}[1][\DefaultGfunction]{\LspaceSymbol^{{#1}}}
\newcommand{\LGastspace}[1][\DefaultGfunction]{\LspaceSymbol^{{#1^\star}}}

\DeclareMathOperator{\EspaceSymbol}{\mathbf{E}} 
\newcommand{\EGspace}[1][\DefaultGfunction]{\EspaceSymbol^{#1}}

\DeclareMathOperator{\WspaceSymbol}{\mathbf{W}} 

\newcommand{\WLGspace}[1][\DefaultGfunction] { {\WspaceSymbol^1}\LspaceSymbol^{{#1}} }
\newcommand{\WLGTspace}[1][\DefaultGfunction]{ {\WspaceSymbol^1_T}\LspaceSymbol^{{#1}}}
\newcommand{\WTLGtildespace}[1][\DefaultGfunction]{{\widetilde{\WspaceSymbol}^1_T}\LspaceSymbol^{{#1}}}

\newcommand{\norm}[1]{\|#1\|}

\newcommand{\LGnorm}[2][\DefaultGfunction]{\norm{#2}_{\LGspace[{#1}]}}
\newcommand{\LGastnorm}[2][{\DefaultGfunction}]{\norm{#2}_{\LGastspace[{#1}]}}
\newcommand{\WLGnorm}[2][\DefaultGfunction]{\norm{#2}_{\WLGspace[{#1}]}}

\newcommand{\Hcal}{\mathcal{H}}
\newcommand{\Hcalast}{\Hcal^\star}
\newcommand{\Lcal}{\mathcal{L}}

\usepackage[margin=2.0cm,marginpar=1.5cm]{geometry}

\usepackage{fancyhdr}
\usepackage[textsize=small,backgroundcolor=white,linecolor=black]{todonotes}

\title{Clarke duality for Hamiltonian systems with nonstandard growth}
\author{
Sonia Acinas \thanks{SECyT-UNRC,  FCEyN-UNLPam and UNSL}\\
Dpto. de Matem\'atica, Facultad de Ciencias Exactas y Naturales\\
Universidad Nacional de La Pampa\\
(L6300CLB) Santa Rosa, La Pampa, Argentina\\
\url{sonia.acinas@gmail.com}\\
Jakub Maksymiuk \\
Faculty of Applied Physics and Mathematics, Gdansk University of Technology,\\
Narutowicza 11/12, 80-233, Gdansk, Poland \\
\url{jakub.maksymiuk@pg.edu.pl}\\
 Fernando D. Mazzone \thanks{SECyT-UNRC, FCEyN-UNLPam and CONICET}\\
Dpto. de Matem\'atica, Facultad de Ciencias Exactas, F\'{\i}sico-Qu\'{\i}micas y Naturales\\
Universidad Nacional de R\'{i}o Cuarto\\
(5800) R\'{\i}o Cuarto, C\'ordoba, Argentina,\\
\url{fmazzone@exa.unrc.edu.ar} \\
}
\date{}

\begin{document}

\maketitle

\begin{abstract}
We consider the existence of periodic solutions to  Hamiltonian Systems with  growth conditions involving G-function. We introduce the notion of symplectic G-function and provide relation for the growth of Hamiltonian in terms of certain constant $C_G$ associated to symplectic G-function $G$. We discuss an optimality of this constant for some special cases. We also provide an applications to the $\Phi$-laplacian type systems.
\end{abstract}

\begingroup
\renewcommand{\thefootnote}{}
\endgroup


\section{Introduction}
In this paper we consider the problem of existence of periodic solutions to the Hamiltonian system
\begin{equation}\label{eq:ham-sis-qp-II}
J\dot{u}=- \nabla \Hcal(t,u(t))
\end{equation}
where the \emph{Hamiltonian} $\Hcal$ is in $C^1([0,T]\times \R^{2n}, \R)$,   $u:[0,T]\to\R^{2n}$ and $J$ denotes the canonical symplectic matrix
\begin{equation*}
J=
\begin{pmatrix}
0_{n\times n}&I_{n\times n}
\\
-I_{n\times n}&0_{n\times n}
\end{pmatrix}
\end{equation*}
Our work is motivated by the book by J. Mawhin and M. Willem \cite{mawhin2010critical} and by the paper by Y. Tian and W. Ge \cite{TiaGe07}.
In \cite[Theorem 3.1]{mawhin2010critical} the authors assume a quadratic growth condition on $\Hcal$:
\begin{equation*}
\Hcal(t,u)\leq \frac{\alpha}{2}|u|^2+\gamma(t),
\end{equation*}
where $\alpha\in (0,2\pi/T)$, $\gamma\in \Lpspace[2]$, and a coercivity condition
$
\lim_{u\to \infty}\frac{1}{T}\int_0^T\Hcal(t,u)\,dt=\infty.
$
Then they obtained, using Clarke dual action method, existence of a T-periodic solution to the equation \eqref{eq:ham-sis-qp-II}. This result is further applied to show existence of periodic solution to the classical Lagrangian system (see \cite[Theorem 3.5]{mawhin2010critical}).

These results was extended in \cite{TiaGe07}, where the same methods are applied to the Hamiltonians of the following form
\begin{equation}\label{eq:hamil_tian}
 \Hcal(t,u)=\frac{1}{a}F(t,u_1)+\frac{a^{q-1}}{q}|u_2|^q,\ \text{$u=(u_1,u_2)$ and $a>0$.}
\end{equation}
The Authors also consider Lagrangian systems. In fact, solutions corresponding to this particular Hamiltonian provide solutions of the p-laplacian equation:
\begin{equation*}
\frac{d}{dt} ( |\dot{u}_1|^{p-2}\dot{u}_1) + \nabla F(t,u_1)=0,\ \frac1p+\frac1q=1.
\end{equation*}
Among other conditions, they assume that $F:[0,T]\times\mathbb{R}^n\to\mathbb{R}$ satisfies the following growth conditions.
There exists $l\in \LGspace[{2\max\{q,p-1\}}]([0,T],\mathbb{R}^n)$ such that
\begin{equation}\label{eq:A1}
 F(t,y)\geq \inner{l}{ |y|^{\frac{p-2}{2}}y},\quad y\in\mathbb{R}^n,\,\text{a.e. } t\in[0,T],\tag{A1}
\end{equation}
and there exists $0<a <\min\{T^{-\frac{p}{q}}, T^{-1}\}$ and $\gamma\in \LGspace[{\max\{q,p-1\}}]([0,T])$ such that
\begin{equation}\label{eq:A2}
 F(t,y)\leq \frac{a^2}{p}|y|^p+\gamma(t),\quad y\in\mathbb{R}^n,\,\text{a.e. } t\in[0,T],\tag{A2}
\end{equation}

The objective of this paper is to extend these results. Our main result, Theorem \ref{thm:solution-ham}, establish existence of solutions for equation \eqref{eq:ham-sis-qp-II} to the case of Hamiltonians with an anisotropic growth conditions given by a G-function $G$. We will seek for solutions in anisotropic Orlicz-Sobolev space (see Section \ref{sec:auxiliary}).

This result improve the results of \cite{TiaGe07} in several directions. First, using anisotropic $G$-functions we can consider more general growth conditions. In particular, we allow $\Hcal$ to have different growth in different directions. Moreover, we do not assume that Hamiltonians have any particular structure like $H(t,u)=H_1(t,u_1)+H_2(t,u_2)$.

Our theorem also improve results of \cite[Theorem 2.1]{TiaGe07}, when Theorem \ref{thm:solution-ham} is applied to the Hamiltonian of the form \eqref{eq:hamil_tian} it provides better result (see Remarks \ref{rem:comp_tian} and \ref{rem:condH1'}).

The method used in \cite{mawhin2010critical,TiaGe07} and in the present paper involves the Clarke dual action functional. It is shown that the critical points of the dual action gives solutions to the problem \eqref{eq:ham-sis-qp-II}. The Clarke duality was introduced in 1978 by F. Clarke \cite{Clarke-1979}, and it was developed by F. Clarke and I. Ekeland in \cite{ Fran-1981, Fran-1982,Iva-1979,book:971200}, to overcome the difficulty that appear when the Hamiltonian action is indefinite. In \cite{Fran-1980}, the Clarke duality was applied to prove some result on the famous Rabinowitz conjecture.

To obtain existence result, we need to prove that the dual action for a perturbed problem with  associated Hamiltonian $\Hcal_{\varepsilon}$, $\varepsilon>0$ small enough, is differentiable and coercive. To do this we introduce in Section \ref{sec:symplectic} the notion of symplectic and semi-symplectic G-function. We show in Section \ref{sec:differentiability} that if the Hamiltonian satisfies
\begin{equation*}
G(\lambda u)-\beta(t)\leq \Hcal (t,u)\leq G(\Lambda u)+\gamma(t),
\end{equation*}
then the associated dual action functional is differentiable on the anisotropic Orlicz-Sobolev space $\WLGspace[G^\star]$, where $G^\star$ denotes the convex conjugate of $G$.

To show that perturbed dual action is coercive we need estimates for the quadratic form
$\int_0^T\inner{J\dot u}{v}\, dt$. We show in Section \ref{sec:symplectic} that for semi-symplectic function $G$ this quadratic form is bounded on Orlicz-Sobolev space $\WLGTspace$ and that
\begin{equation*}
\int_0^T\inner{J\dot u}{u}\, dt \geq -C_1\int_0^T G(T\dot{u})\, dt-C_2
\end{equation*}
on $\WLGTspace$.

It turns out that the constant $C_1$ is related to the growth condition on Hamiltonian that we consider in Theorem \ref{thm:solution-ham}:
\begin{equation*}
\Hcal(t,u)\leq G(\Lambda u)+\gamma(t),
\end{equation*}
where $\Lambda^{-1}> T \max\{1,C_{G}\}$ and $\gamma\in \Lpspace[1]$. Namely, the smaller value of $C_1$ gives the wider class of Hamiltonians we can consider. Therefore, it is important to determine the optimal value for $C_1$ (we denote it by $C_{G}$). We show that this optimal value is related to certain constrained optimization problem and we obtain the optimal value for $C_{G}$ in some simple cases. In Section \ref{sec:symplectic} we also discuss how the constant $C_{G}$ and the given bound for $\Lambda$ are related to the bounds for $\alpha$ imposed in \cite{mawhin2010critical,TiaGe07}.

This paper is structured as follows. In Section \ref{sec:auxiliary} we present the auxiliary results. We briefly recall the notion of G-function and Orlicz-Sobolev spaces. In Section \ref{sec:symplectic} we introduce the concept of symplectic G-function and we study some properties. The main result about symplectic G-function is Theorem \ref{thm:LowBoundQuadraticForm} which establishes boundedness of certain canonical quadratic functional. In Section \ref{sec:differentiability}, we discuss  differentiability of the dual action. In Section \ref{sec:existence}, we present our main result,  which establishes existence of periodic solutions for Hamiltonian system. Finally, in Section \ref{sec:applic} we apply the previous results to the problem of existence solutions for certain second order systems.

\section{Auxiliary results}
\label{sec:auxiliary}

In this section we collect some auxiliary results. First, we briefly recall some facts concerning convex function. Next, we will be concerned with the notion of G-function and Orlicz spaces. We refer the reader to \cite{HirLem01,mawhin2010critical} for more comprehensive information about convex functions and to \cite{MA2017,BarCia17,chamra2017anisotropic,Sch05,trudinger1974imbedding} for more information on anisotropic G-functions and Orlicz spaces.


\subsection{Convex functions}

Recall that for arbitrary convex function $G\colon \R^n\to \R$ the convex conjugate of $G$ is defined by
\begin{equation*}
G^\star\colon \R^n\to (-\infty,\infty],\quad    G^\star(v)=\sup_{u\in \R^n}\{\inner{u}{v}-G(u)\}.
\end{equation*}

In general, $G^\star$ need not be finite. Assuming $\lim_{|u|\to \infty}\frac{G(u)}{|u|}=\infty$ we get $G^\star<+\infty$. Immediately from the definition of $G^\star$ we get:
\begin{itemize}
\item $G_1\leq G_2 \implies G_2^\star\leq G_1^\star$,
\item $F(u)=a G(bu)-c \implies F^\star(v)=aG^\star(v/ab)+c$, where $a,b>0$ and $c \in \R$,
\item \emph{Fenchel's inequality:}
\begin{equation*}
    \inner{u}{v}\leq G(u)+G^\star(v),
\end{equation*}
\item let $G_i\colon \R^{n_i}\to \R$, $i=1,2$, be continuous convex functions. Define $F\colon
\R^{n_1}\times\R^{n_2}\to \R$ by
\begin{equation*}
    F(u)=F(u_1,u_2)=G_1(u_1)+G_2(u_2)
\end{equation*}
then
\begin{equation*}
    F^\star(v)=F^\star(v_1,v_2)=G^\star_1(v_1)+G^\star_2(v_2),
\end{equation*}
where the inner product in $\R^{n_1}\times\R^{n_2}$ is taken as the sum of inner products in components,

\item  if $G$ is a differentiable convex function, then
\begin{equation}
    \label{ineq:G_by_gradientG}
    G(u_1)-G(u_1-u_2)\leq \inner{\nabla G(u_1)}{u_2}\leq G(u_1+u_2)-G(u_1)
\end{equation}
\item \emph{Young's identity:} if $G$ is a differentiable convex function, then
\begin{equation}
\label{eq:Fenchel:equality}
    \inner{\nabla G(u)}{u}=G(u)+G^\star(\nabla G(u)).
\end{equation}
\end{itemize}

\begin{definition}
Let us define $\Gamma(\R^n)$ to be the set of all differentiable, strictly convex functions $G\colon \R^n\to \R$ such that
\begin{equation}\label{eq:cond-G-infinity}
\lim_{|u|\to \infty}\frac{G(u)}{|u|}=\infty.
\end{equation}
\end{definition}
It is well known that if $G$ is in  $\Gamma(\R^n)$ then its convex conjugate is also in $\Gamma(\R^n)$. Moreover, in this case relation $\nabla G^\star = (\nabla G)^{-1}$ holds.
The next lemma is a generalization of Proposition 2.2 from \cite{mawhin2010critical}.

\begin{proposition}\label{prop: cota-conj-phi}
Let $H\colon \R^n\to \R$ be a differentiable convex function. Assume that there exists a convex function $G\colon \R^n\to \R$ satisfying
\eqref{eq:cond-G-infinity} and constants $\beta$, $\gamma>0$ such that
\begin{equation}\label{eq:cotas-F-aniso}
-\beta \leq H(u) \leq G(u)+\gamma, \mbox{ for all } u \in \R^n.
\end{equation}
Then  for any $r>1$
\begin{equation}\label{eq:trans-grad-aniso}
G^\star(\nabla H (u))\leq \frac{1}{r-1}G(r u)+\frac{r}{r-1}\left(\beta+\gamma\right).
\end{equation}
\end{proposition}
\begin{proof}
Conjugating \eqref{eq:cotas-F-aniso} and using \eqref{eq:Fenchel:equality}, we obtain
\begin{equation*}
G^\star(\nabla H(u))-\gamma \leq H^\star(\nabla H(u))=\inner{\nabla H(u)}{u}-H(u).
\end{equation*}
From Fenchel's inequality, we get
\begin{equation*}
\inner{\nabla H(u)}{u} =
\frac{1}{r} \inner{\nabla H(u)}{r u}
\leq
\frac{1}{r}G^\star(\nabla H(u))+\frac{1}{r}G(r u).
\end{equation*}
Combining the above inequalities and \eqref{eq:cotas-F-aniso} we obtain
\begin{equation*}
G^\star(\nabla H(u)) \leq \frac{1}{r}G^\star(\nabla H(u))+\frac{1}{r}G(r u) + \beta+\gamma,
\end{equation*}
which implies \eqref{eq:trans-grad-aniso}.
\end{proof}

\begin{remark}
\label{com:mejora_desi}
Inequality \eqref{eq:trans-grad-aniso} for the case of the power function $G(u)=|u|^q $ is slightly
better than the corresponding inequality in \cite[Proposition 2.2]{mawhin2010critical} where the estimate $|\nabla H(v)|\leq C(|u|+\beta+\gamma+1)^{q-1}$ is obtained. Here we obtain $|\nabla H(v)|\leq C(|u|^q+\beta+\gamma)^{(q-1)/q}$. This simple fact allows us in forthcoming results to use less restrictive hypothesis on certain functions.
\end{remark}


\subsection{G-functions and Orlicz spaces}

\begin{definition}
A function $G\colon \R^n\to [0,+\infty)$ is called a G-function if $G$ is convex and satisfies $G(x)=0\iff x=0$, $G(-u)=G(u)$ and $\lim_{|u|\to\infty}\frac{G(u)}{|u|}=\infty$.
\end{definition}
It follows that the convex conjugate of a G-function is also a G-function.

\begin{proposition}
\label{prop:elem_ine}
Let $G$ be a G-function. Then, for every $u\in \R^n$ we have
\begin{align*}
0<s_1\leq s_2 &\implies s_2G(u/s_2)\leq s_1G(u/s_1),\\
0<s_1, s_2 &\implies G(s_1u)+G(s_2u)\leq G((s_1+s_2)u).
\end{align*}
\end{proposition}
The proof are straightforward.
Immediately from the Fenchel inequality we get that for every $\mu,\nu>0$ and every $u,v\in\R^n$
\begin{equation}
    \label{ineq:generalFenshel}
    -\mu\nu\,G(u/\mu)-\mu\nu\,G^\star(v/\nu)
    \leq \inner{u}{v}\leq
    \mu\nu\,G(u/\mu)+\mu\nu\,G^\star(v/\nu),
\end{equation}
since $G(-u)=G(u)$.

We say that  G-function $G$ satisfies the  \emph{$\Delta_2$-condition} (denoted $G \in \Delta_2$), if there exists a constant $C>0$  such that for every $u\in\R^n$
\begin{equation}\label{eq:delta2defi}
G(2u)\leq C G(u)+1.
\end{equation}
Note that this definition is equivalent that the traditional one, i.e. that there exixts $r_0,C>0$ with $G(2u)\leq C G(u)$ for $|u|>r_0$. If there exists $C>0$ such that $G(2u)\leq CG(u)$ for all $u\in\R^{n}$, then we say that $G$ satisfies the \emph{$\Delta_2$-condition globally}.

Recall that $G_1\prec G_2$ if there exist $K>0$ and $C\geq 0$ such that $G_1(u)\leq G_2(Ku)+C$, for every $u\in\R^n$. Directly from the definition, if $G_1\prec G_2$ then $G^\star_2\prec G^\star_1$.

Let $G$ be a G-function. The Orlicz space $\LGspace=\LGspace([0,T],\R^n)$ is defined to be
\begin{equation*}
\LGspace = \left\{u\colon [0,T]\to \R^n\colon \text{ $u$-measurable },
\exists\lambda>0\, \int_0^T G(\lambda u) \,dt<\infty\right\}.
\end{equation*}
The space $\LGspace$ equipped with the Luxemburg norm
\begin{equation*}
\LGnorm{u}=\inf\left\{\lambda>0\colon \int_0^T G(u/\lambda) \,dt \leq 1\right\}
\end{equation*}
is a Banach space. Observe that
\begin{equation*}\label{eq:lowbound_for_modular_by_norm}
\LGnorm{u}>1 \implies \int_0^T G(u) \,dt \geq \LGnorm{u}
\end{equation*}
and therefore for any $u\in\LGspace$
\begin{equation}\label{eq:lowbound_for_modular_by_norm2}
 \LGnorm{u}\leq \int_0^T G(u) \,dt +1.
\end{equation}
A generalized form of Holder's inequality holds
\begin{equation*}\label{ineq:Holder}
\int_0^T \inner{u}{v} \,dt \leq 2 \LGnorm{u}\LGastnorm{v},\quad u\in \LGspace, v\in \LGastspace.
\end{equation*}

The subspace $\EGspace=\EGspace([0,T],\R^n)$ is defined to be the closure of $\Lpspace[\infty]$ in $\LGspace$. The equality $\EGspace=\LGspace$ holds if and only if $G\in \Delta_2$. It is known that $\EGspace[G^\star]$ is separable and $\LGspace = \left(\EGspace[G^\star]\right)^\star$. Hence $\LGspace$ can be equipped with weak$\star$ topology induced from $\EGspace[G^\star]$.

We define the anisotropic Orlicz-Sobolev space of vector valued functions $\WLGspace=\WLGspace([0,T],\R^n)$ by
\begin{equation*}
\WLGspace=\{u\in \LGspace\colon \dot{u} \text{ absolutely continuous and } \dot{u}\in \LGspace\}.
\end{equation*}
The space $\WLGspace$ is a Banach space when equipped with the norm
\begin{equation*}
\WLGnorm{u}=\LGnorm{u}+\LGnorm{\dot{u}}.
\end{equation*}

As usual, for a function $u\in \Lpspace[1]([0,T],\R^n)$ we will write $u=\widetilde{u}+\overline{u}$, where $\overline{u}=\frac{1}{T} \int_0^T u \,dt$. One can show that
\begin{equation}\label{eq:WLG_equivalent_norm}
\WLGnorm{u}'=|\overline{u}|+\LGnorm{\dot{u}}
\end{equation}
is an equivalent norm to $ \WLGnorm{\cdot}$ (see \cite[Remark 1]{MA2017}). We set $\WLGTspace:=\left\{u\in \WLGspace: u(0)=u(T)\right\}$ and $\WTLGtildespace=\left\{v \in \WLGTspace: \int_0^T v(t)\, dt=0\right\}$.

In the space $\WLGspace$ an anisotropic version of Poincar\'e-Wirtinger inequality holds (see \cite{MA2017} or \cite{chamra2017anisotropic}):
\begin{equation*}
    G(\widetilde{u})\leq \frac{1}{T} \int_0^T G(T\dot{u}) \,dt.
\end{equation*}
Integrating both sides, we get

\begin{equation}
    \label{ineq:conjugateP-W}
    \int_0^T G(\widetilde{u}) \,dt \leq \int_0^T G( T\, \dot{u}) \,dt.
\end{equation}

We will also use the following simple lemma.

\begin{lemma}[{see \cite[Corollary 2.5]{MA2017}}]
\label{lem:bounded_has_unif_conv_seq}
If $u_k$ is a bounded sequence in Orlicz-Sobolev space then $u_k$ has a uniformly convergent
subsequence.
\end{lemma}


\section{Symplectic G-functions}
\label{sec:symplectic}

\begin{definition} We say that a G-function $G\colon \R^{2n}\to [0,\infty)$  is \emph{symplectic} if
$G^\star(Ju)=G(u)$ for all $u\in \R^{2n}$.
\end{definition}

It is obvious that if $G$ is symplectic then $G^\star$ is also symplectic. On the other hand, if a symplectic function satisfies $\Delta_2$-condition then its conjugate also satisfies this condition. Note that if a G-function $G$ is differentiable and symplectic,  then $G^\star$ is also differentiable and
\begin{equation}\label{eq:nabla-sym}
\nabla G(u)= J\nabla G^\star(Ju).
\end{equation}

\begin{definition}
We say that a  G-function $G\colon \R^{2n}\to [0,\infty)$ is \emph{semi-symplectic} if
$G^\star(J\cdot)\prec G$.
\end{definition}

Obviously, every symplectic function is semi-symplectic.

\begin{example} If $\Phi:\R^n\to [0,+\infty)$ is a G-function, then $G(u_1,u_2)=\Phi(u_1)+\Phi^\star(u_2)$ is  symplectic. A typical example of such a function is $G(u_1,u_2)=|u_1|^p+|u_2|^q$, $1/p+1/q=1$.

If $\Phi_1,\Phi_2\colon \R^n\to [0,\infty)$ satisfy $\Phi_1^\star\prec \Phi_2$ then the function of the form
\begin{equation*}
G(u)=\Phi_1(u_1)+\Phi_2(u_2),
\end{equation*}
is semi-symplectic but not necessary symplectic.

A more involved example is provided by $F(u)=G(Au)$, where $G$ is a  symplectic G-function and $A$ is a symplectic matrix, i.e. $A^{-T}J=JA$. In order to prove the symplecticity of $F$, note that $F^\star(v)=G^\star(A^{-T}v)$. Consequently, $F^\star(J\cdot)=G^\star(A^{-T}J\cdot)=G^\star(JA\cdot)=G(A\cdot)=F(\cdot)$. In this way,  we can produce more examples of symplectic G-functions than those given previously. For example,
\begin{equation*}
G(u_1,u_2)=\Phi(u_1+u_2)+\Phi^\star(u_1+2u_2)
\end{equation*}
is a symplectic G-function.
\end{example}

Note that the Orlicz space generated by the function $G(u_1,u_2)=\Phi(u_1)+\Phi^\star(u_2)$ is a product of Orlicz spaces $\LGspace[\Phi]$ and $\LGspace[\Phi^\star]$. This is exactly the case considered in \cite{TiaGe07} (see the definition of the space $X$ therein). However, the Orlicz space corresponding to $G(u_1,u_2)=\Phi(u_1+u_2)+\Phi^\star(u_1+2u_2)$ is not the product of Orlicz spaces (cf. \cite[Example 3.7]{chamra2017anisotropic}).

\begin{proposition}
    \label{prop:embbeding_for_symplectic_G}
If $G$ is semi-symplectic then $J$ induces an embedding
\begin{equation*}
u\mapsto Ju,\quad \LGspace([0,T],\R^{2n})\hookrightarrow \LGastspace([0,T],\R^{2n}).
\end{equation*}
Moreover, for any $K>0$, $C\geq 0$ such that $G^\star(Ju)\leq G(Ku)+C$, for all $u\in \R^{2n}$, we
have
\begin{equation*}
\LGastnorm{Ju}\leq K(C\,T+1)\LGnorm{u}.
\end{equation*}
\end{proposition}
\begin{proof}
Fix $K>0$, $C\geq 0$ such that $G^\star(Ju)\leq G(Ku)+C$, for all $u\in \R^{2n}$. Let $u\in
\LGspace$. Then there exists $\lambda >0$ such that $ \int_0^T G(u/\lambda) \,dt<\infty$
and
\begin{equation*}
\int_0^T G^\star\left( \frac{Ju}{K\lambda} \right) \,dt
\leq
C\,T + \int_0^T G(u/\lambda ) \,dt<\infty.
\end{equation*}
So that $Ju\in \LGastspace$. Suppose that $\LGnorm{u}=1$.  Then $ \int_0^T G(u) \,dt\leq 1$ and hence
\begin{equation*}
\int_0^T G^\star\left( \frac{Ju}{K(C\,T+1)} \right) \,dt
\leq
\frac{1}{C\, T + 1} \int_0^T G^\star\left( \frac{Ju}{K} \right) \,dt \leq 1.
\end{equation*}
This inequality implies that
\begin{equation*}
\LGastnorm{Ju}\leq K(C\,T+1)
\end{equation*}
and the result follows.
\end{proof}

Let $G\colon \R^{2n}\to [0,\infty)$ be a semi-symplectic G-function. From  Proposition \ref{prop:embbeding_for_symplectic_G}, it follows that the bilinear form
\begin{equation}\label{eq:bilinear-form}
 \int_0^T \inner{Jv}{u} \,dt,
\end{equation}
is well-defined and it is bounded on $\LGspace([0,T],\R^{2n})\times \LGspace([0,T],\R^{2n})$.

It is proved in \cite[Propossition 3.2]{mawhin2010critical} that for every $u\in
\WspaceSymbol_T^{1,2}([0,T],\R^{2n})$
\begin{equation}\label{eq:CG2_constant}
    \int_0^T \inner{J\dot u}{u}\, dt \geq -\frac{T}{2\pi} \int_0^T|\dot u|^2\,dt
\end{equation}
Similar estimate was obtained in \cite{TiaGe07} for $G(u_1,u_2)=|u_1|^p/p+|u_2|^q/q$, $1/p+1/q=1$. Below we show that the analogous estimate can be obtained for Orlicz-Sobolev space induced by any semi-symplectic G-function.

\begin{theorem}\label{thm:LowBoundQuadraticForm}
Let $G$ be a semi-symplectic G-function. Then there exist  constants $C_1$, $C_2>$ depending only
on $G$ and $T$ such that for every $u \in \WLGTspace([0,T],\R^{2n})$ we have
\begin{equation}\label{eq:cotaJu}
\int_0^T \inner{J\dot{u}}{u}\,dt\geq -C_1\int_0^T G\left(T\dot{u}\right)\,dt-C_2.
\end{equation}
\end{theorem}
\begin{proof}
Let $u\in\WLGTspace([0,T],\R^{2n})$ and fix $K>0$, $C\geq 0$ such that $G^\star(Ju)\leq G(Ku)+C$, for all $u\in \R^{2n}$. By the Fenchel's inequality \eqref{ineq:generalFenshel}, the fact that  $G$ is a semi-symplectic and inequality \eqref{ineq:conjugateP-W}, we obtain
\begin{multline*}
\int_0^T \inner{J\dot{u}}{u} \,dt
=
\frac{K}{T}\int_0^T \inner{\frac{T}{K} J\dot{u}}{\widetilde{u}} \,dt
\geq\\\geq
-\frac{K}{T}\left\{\int_0^T G^\star \left(J\frac{T\dot{u}}{K}\right) \,dt+\int_0^T G(\widetilde{u}) \,dt\right\}\geq
-\frac{K}{T}\left\{ 2\int_0^T G(T\dot{u})\,dt+C\right\}.
\end{multline*}
\end{proof}

If $G$ is symplectic, instead of semi-symplectic, following the same lines as the proof of Theorem \ref{thm:LowBoundQuadraticForm}, we can prove that inequality \eqref{eq:cotaJu} is satisfied  with $C_1(T)=2/T$ and $C_2=0$. In addition, after the change of variable $t=Ts$, inequality \eqref{eq:cotaJu} takes the form
\begin{equation*}
\int_0^1\inner{J\dot{u}}{u}dt\geq -2\int_0^1 G\left(\dot{u}\right) dt.
\end{equation*}

The value  of the constant $C_1$  in  Theorem \ref{thm:LowBoundQuadraticForm} imposes restrictions on the results obtained in the following sections. A smaller constant $C_1$ results in a  more inclusive estimate for $\Lambda$ in Theorem \ref{thm:solution-ham} and Proposition \ref{prop:optimal_lambda}. Therefore, it is useful to obtain the smallest possible value for $C_1$.  For example, in \cite{mawhin2010critical} it is proved that $C_1=1/\pi$ when $G(u)=|u|^2/2$. In this case, we can see that the optimal constant  is far from $2$.

\begin{definition}
For a symplectic G-function $G$ we define
\begin{equation}\label{eq:def-optimal-constant}
C_{G}(T)=-\inf\left\{
\frac{\displaystyle\int_0^T \inner{J\dot{u}}{u }\,dt}{\displaystyle\int_0^T G\left( T\dot{u}\right)\,dt}
:\; u \in  \WLGTspace([0,T],\R^{2n})
\right\}
\end{equation}
\end{definition}
The rest of this section is devoted to the problem of optimality of $C_{G}(T)$. We relate this problem to the constrained optimization  problem and we obtain exact values in some special cases. Note that  the change of variable $t=Ts$ implies that $C_{G}(T)=C_{G}(1)/T$. Therefore, from now on in this section  we will assume that $T=1$ and $G$ is a symplectic function. For simplicity, we put $C_{G}:=C_{G}(1)$.

\begin{proposition}
The relation $C_{G}=C_{G^\star}$ holds for every symplectic function $G$.
\end{proposition}
\begin{proof}
For $v\in \WLGTspace[G^\star]$ and $u=Jv$, we have
\begin{equation*}
\frac{\displaystyle\int_0^1\langle J\dot{u},u\rangle dt}{\displaystyle\int_0^1G(\dot{u})dt}
=
\frac{\displaystyle\int_0^1\langle -\dot{v},Jv\rangle dt}{\displaystyle\int_0^1G(J\dot{v})dt}
=\frac{\displaystyle\int_0^1\langle J\dot{v},v\rangle dt}{\displaystyle\int_0^1G^\star(\dot{v})dt}
\end{equation*}
Using the fact that $u\mapsto Ju$ is invertible from $\WLGTspace([0,1],\R^{2n})$ onto $\WLGTspace[G^\star]([0,1],\R^{2n})$, the statement follows.
\end{proof}

Consider the following constrained optimization problem on $\WLGTspace\left([0,1],\R^{2n}\right)$:
\begin{equation}
\label{eq:eig_var_prob}
\tag{P}
\left\{
\begin{array}{l}
\hbox{minimize }f(u)\\
\hbox{subject to } g(u)=\gamma\\
\end{array}
\right.
\end{equation}
where $f,g:\WLGTspace\to \R$ are given by
\begin{equation*} f(u)=\int_0^1 \langle J\dot{u},  u \rangle \,dt,\quad
 g(u)=\int_0^1 G\left( \dot{u}\right)\,dt.
\end{equation*}
It is obvious that $f$ is $C^1$ map. Moreover, if $G^\star$ satisfies $\Delta_2$ then $g$ is also $C^1$ map.

For $\gamma>0$ set
\begin{equation*}
A(\gamma) = \inf\left\{f(u) \colon\;  u \in  \WLGTspace\left([0,1],\R^{2n}\right),\quad g(u)=\gamma \right\}.
\end{equation*}
With this notation we have
\begin{equation}\label{eq:CG}
C_{G}=-\inf\limits_{\gamma>0}\gamma^{-1} A(\gamma).
\end{equation}

\begin{lemma}\label{lem:exist}
Assume that $G$, $G^\star\in\Delta_2$. The problem \eqref{eq:eig_var_prob} has a solution $u_{\gamma}\in \WLGTspace\left([0,1],\R^{2n}\right)$.
\end{lemma}
\begin{proof}
Note that if $u(t)$ is an admissible function for the problem \eqref{eq:eig_var_prob} (i.e. $u\in \WLGTspace\left([0,1],\R^{2n}\right)$ and $g(u)=\gamma$) then $v(t)=u(1-t)$ is also admissible. Hence $f(u)$ and $f(v)$ have different sign and consequently $A(\gamma)<0$.

Let $u_n$ be a minimizing  sequence for \eqref{eq:eig_var_prob}. We can assume that $f(u_n)<0$. Since $f(u+c)=f(u)$ for every $c\in\R^{2n}$, we can suppose that $\overline{u}=0$. It follows that $u_n$ is bounded.

This implies that there exists a subsequence (denoted $u_n$ again) and $u_{\gamma}\in\WLGTspace$ such that $u_u\to u_{\gamma}$ uniformly and $\dot{u}_u\rightharpoonup \dot{u}_{\gamma}$. 
Thus, by definition, $A(\gamma) = f(u_{\gamma})$. Since  $A(\gamma)<0$, we have $\dot{u}_{\gamma}\neq  0$ and $g(u_{\gamma})>0$. Since $g$ is weakly lsc, we have that $g(u_{\gamma})\leq \gamma$.

If $g(u_{\gamma})<\gamma$, then there would be a $\lambda>1$ with $g(\lambda  u_{\gamma})=\gamma$. But then $f(\lambda u_{\gamma}) =\lambda^2f(u_{\gamma})< f(u_{\gamma})=A(\gamma)$ which is a
 contradiction. This implies that $u_{\gamma}$ is admissible and the proof is finished.

\end{proof}

\begin{theorem}\label{thm,charact-constant}
Let $G$ be a differentiable and strictly convex symplectic function satisfying $\Delta_2$ condition. Then
\begin{equation}
\label{eq:Agamma}
C_{G}=\sup\frac{1}{ T_u}\frac{\displaystyle\int_0^{T_u} \inner{\nabla G(u)}{u}dt}{\displaystyle\int_0^{T_u} G^\star(\nabla G(u))dt},
\end{equation}
where the supremum is taken among all periodic solutions of the Hamiltonian system $J\dot{u}(t)=-\nabla G(u(t))$ and the constant $T_u$ denotes a period of $u$.
\end{theorem}

\begin{proof}
Using Lemma \ref{lem:exist}, we obtain a function $u_{\gamma}\in \WLGTspace\left([0,1],\R^{2n}\right)$  satisfying  $f(u_{\gamma})=A(\gamma)$. Applying the Lagrange multiplier rule, we find $\lambda\in\R$ such that
\begin{equation*}
f'(u_{\gamma})=2\lambda g'(u_{\gamma}).
\end{equation*}
Consequently, for any $ w\in \WLGTspace([0,1],\R^{2n})$  we have that
\begin{equation*}
0=2\int_0^1 \inner{J\dot{u}_{\gamma}}{ w}\,dt- 2\lambda  \int_0^1 \inner{\nabla G(\dot{u}_{\gamma})}{ \dot{w}}\,dt.
\end{equation*}
Integrating by parts we get
\begin{equation}\label{eq:derivative-before-subst}
0=\int_0^1 \inner{J\dot{u}_{\gamma}+\lambda  \frac{d}{dt} \nabla G\left(\dot{u}_{\gamma}\right) }{w} \,dt
-\lambda \inner{\nabla G\left(\dot{u}_{\gamma}\right)}{ w}\Big|_0^1
\end{equation}
We deduce that for every $w \in C_0^{\infty}([0,1],\R^{2n})$
\begin{equation*}
0=\int_0^1 \left\langle
J\dot{u}_{\gamma}+\lambda \frac{d}{dt} \nabla G\left(\dot{u}_{\gamma}\right)
, w\right\rangle \,dt
\end{equation*}
Hence $J\dot{u}_{\gamma}+\lambda \frac{d}{dt} \nabla G\left(\dot{u}_{\gamma}\right)=0$ a.e. $t \in [0,1]$.
Consequently,  \eqref{eq:derivative-before-subst} implies that for every  $w \in \WLGTspace([0,1],\R^{2n})$
\begin{equation*}
0=\inner{\nabla G\left(\dot{u}_{\gamma}\right)}{ w}\Big|_0^1=
\left\langle \nabla G\left(\dot{u}_{\gamma}
(1)\right)-
\nabla G\left(\dot{u}_{\gamma}(0)\right), w(0)\right\rangle.
\end{equation*}
Since $w(0)$ is arbitrary and  $\nabla G$ is a one-to-one map, we have  $\dot{u}_{\gamma}(1)=\dot{u}_{\gamma}(0)$. Hence, $u_{\gamma}$ solves
\begin{equation}\label{eq:hamiltonianJ}
\left\{
\begin{array}{l}
J\dot{u}_{\gamma}+\lambda \frac{d}{dt} \nabla G(\dot{u}_{\gamma})=0, \;\;\mbox{a.e } t \in [0,1]
\\
u_{\gamma}(0)-u_{\gamma}(1)=\dot{u}_{\gamma}(0)-\dot{u}_{\gamma}(1)=0.
\end{array}
\right.
\end{equation}

Integration by parts and \eqref{eq:hamiltonianJ} yields
\begin{equation*}
  A(\gamma)=f(u_\gamma)= \int_0^1\inner{J\dot{u}_\gamma}{u_\gamma}dt
  =-\lambda \int_0^1\inner{\frac{d}{dt}\nabla G(\dot{u}_\gamma)}{u_\gamma}dt
  =\lambda \int_0^1\inner{\nabla G(\dot{u}_\gamma)}{\dot{u}_\gamma}dt.
\end{equation*}
Since $A(\gamma)<0$ (see proof of Lemma \ref{lem:exist}) and $\inner{\nabla G(\dot{u}_\gamma)}{\dot{u}_\gamma}>0$, we get $\lambda<0$.

Define $u(s):=J\nabla  G\left(\frac{du_{\gamma}}{dt}|_{t=\lambda s}\right)$. Note that $u(s)=-\nabla G^\star\left(J\frac{du_{\gamma}}{dt}|_{t=\lambda s}\right)$ by \eqref{eq:nabla-sym}. We have
\begin{multline*}
\nabla G(u(s))
= \nabla G\left(-\nabla G^\star\left(J\frac{du_{\gamma}}{dt}|_{t=\lambda s}\right) \right)
= -\left. J\frac{du_{\gamma}}{dt}\right|_{t=\lambda s}
=\\= \lambda\frac{d}{dt} \nabla  G\left(\frac{du_{\gamma}}{dt}|_{t=\lambda s}\right)
= -\lambda J\frac{d}{dt}u(s)
= - J\frac{d}{ds}u(s)
\end{multline*}
Hence $u$ solves $Jdu/ds=-\nabla G(u(s))$. Since $u$ solves an autonomous system  and $u(0)=u(\lambda^{-1})$, the function $u(s)$ is defined for every $s\in\R$ and is  $T_u$-periodic with $T_u=-\lambda^{-1}$.

Performing the change of variable $t=\lambda s$ we obtain
\begin{multline*}
A(\gamma)
=-\lambda^2 \int_{\lambda^{-1}}^0
    \inner{\nabla  G\left(\frac{du_{\gamma}}{dt}|_{t=\lambda s}\right)}{\frac{du_{\gamma}}{dt}|_{t=\lambda s}}ds
= -\lambda^2 \int_{\lambda^{-1}}^0\inner{Ju(s)}{J\nabla  G(u(s))}ds
=\\=
-\lambda^2 \int_{\lambda^{-1}}^0\inner{u(s)}{\nabla  G(u(s))}ds
=-\frac{1}{T_u^2} \int_{0}^{T_u}\inner{u(s)}{\nabla  G(u(s))}ds
\end{multline*}

Using the fact that $\nabla G(u(s))= -\left. J\frac{du_{\gamma}}{dt}\right|_{t=\lambda s}$ and that $G$ is symplectic we obtain
\begin{equation*}
\gamma= \int_0^1  G(\dot{u}_\gamma)dt
=-\lambda \int_{\lambda^{-1}}^0 G\left(\frac{du_{\gamma}}{dt}|_{t=\lambda s}\right)ds
=-\lambda \int_{\lambda^{-1}}^0 G(J\nabla  G(u(s)))ds
=\frac{1}{T_u}\int_0^{T_u}  G^\star(\nabla  G(u(s)))ds.
\end{equation*}

Thus, we have just proved that for every $\gamma>0$ there exists a $T_u$-periodic function $u$  such that  $\dot{u}=-\nabla G(u)$ and
\begin{equation*}
\frac{A(\gamma)}{\gamma}=
- \frac{1}{T_u} \frac{\displaystyle\int_0^{T_u}\inner{u(s)}{\nabla  G(u(s))}ds }{\displaystyle\int_0^{T_u}  G^\star(\nabla  G(u(s)))ds }
\end{equation*}

On the other hand, let $u:\R\to \R^{2n}$ be a periodic solution of $J\dot{u}(s)=-\nabla G(u(s))$ and let $T_u$ be a period of $u$. Set $u_0(t)=T_u^{-1}u(T_ut)$. Then $u_0\in \WLGTspace\left([0,1],\R^{2n}\right)$ and
\begin{equation*}
\inf_{\gamma>0}\frac{A(\gamma)}{\gamma}=
-C_G\leq \frac{\displaystyle\int_0^{1} \inner{J\dot{u}_0}{u_0} dt }{\displaystyle\int_0^{1}G(\dot{u}_0)dt}=\frac{1}{T_u}\frac{\displaystyle\int_0^{1} \inner{J\dot{u}(T_ut)}{u(T_ut)} dt }{\displaystyle\int_0^{1}G(\dot{u}(T_ut))dt}=-\frac{1}{ T_u}\frac{\displaystyle\int_0^{T_u} \inner{\nabla G(u)}{u}dt}{\displaystyle\int_0^{T_u} G^\star(\nabla G(u))dt},
\end{equation*}

From this assertion  we obtain the desired result.
\end{proof}

\begin{example} If $ G(u)=|u|^2/2$, then the equation $J\dot{u}=-\nabla G(u)$ is equivalent to the harmonic oscillator equation $\ddot{v}+v=0$. Here, we have that $T_u=2k\pi$ with $k\in\mathbb{N}$ and for every $u$. On the other hand, $\inner{\nabla  G (u)}{u}=2 G^\star(\nabla  G(u))$. Therefore $C_{G}=1/\pi $ (cf. \cite[Proposition 3.2]{mawhin2010critical}).
\end{example}

Let us adapt to anisotropic G-functions the definition of Simonenko indices (see \cite{FK97}, cf. \cite{BarCia17, ChmMak18}) :
\begin{equation*}
 p( G)=\inf\limits_{u\neq 0}\frac{\inner{u}{\nabla G(u)}}{G(u)},\quad  q( G)=\sup\limits_{u\neq 0}\frac{\inner{u}{\nabla G(u)}}{G(u)}
\end{equation*}
It is known that $ q( G)<\infty$ if and only if $ G$ is globally $\Delta_2$ and $p(G)>1$ if and only if $ G^\star$ is globally $\Delta_2$ (see \cite[Theorem 5.1]{KR}). Note that   if we write $v=\nabla G(u)$ then
\begin{equation*}
\frac{\inner{u}{\nabla  G(u)} }{ G^\star(\nabla  G(u))}=\frac{\inner{v}{\nabla  G^\star(v)} }{  G^\star(v)}\leq q( G^\star).
\end{equation*}
On the other hand, if $ G$ is symplectic then $p( G^\star)=p( G)$ and $q( G^\star)=q( G)$. The previous reasoning proves the next result.

\begin{corollary}
    \label{eq:estimate_simo}
Let $G$ be a differentiable and strictly convex symplectic function satisfying the $\Delta_2$ condition globally, then
\begin{equation*}
  p( G) (\inf T_u)^{-1} \leq C_{ G} \leq q( G)(\inf T_u)^{-1},
\end{equation*}
where the infimum is taken among all periods of functions $u$ which solve the Hamiltonian system $J\dot{u}=-\nabla G(u)$.
\end{corollary}

Next, we apply the previous results to some particular symplectic function $ G$.

\begin{theorem}
 Suppose that $n=1$ and $G:\mathbb{R}\times\mathbb{R}\to[0,\infty)$ given by  $G(u_1,u_2)=|u_1|^{p}/p+|u_2|^q/q$, with $1<p<\infty$ and $q=p/(p-1)$. Then
  \begin{equation*}
 C_ G= \frac{p\sin\left(\frac{\pi}{p}\right)}{2\pi(p-1)^{1/p}}.
\end{equation*}
 \end{theorem}

\begin{proof}
  It is easy to see that the equation $J\dot{u}=-\nabla G(u)$ is equivalent to  $p$-Laplacian equation
\begin{equation}\label{eq:p.laplacian}
\frac{d}{ds}|\dot{u_1}(s)|^{p-2}\dot{u_1}(s)+ |u_1(s)|^{p-2}{u_1(s)}=0, \;\; s\in\R.
\end{equation}
It is well known that the  1-dimensional $p$-Laplacian equation is  isochronous,
i.e. all solutions are periodic with the same minimal period given by
\begin{equation*}
 T_{p}=4(p-1)^{-1/q}B\left(1+\frac{1}{q},\frac{1}{p}\right)=\frac{4  \pi (p-1)^{1/p} }{p\sin \left(\frac{\pi}{p}\right)},
\end{equation*}
where $B$ denotes the Beta Function (see \cite{AGMS} for the proof).

If $u$ is a  solution of the equation $J\dot{u}=-\nabla   G(u)$,  then for every $\lambda>0$ the function $\overline{u}=(\lambda u_1,\lambda^{p-1}u_2)$ is also a solution. This  observation implies that the quotient
\begin{equation*}\frac{1}{ T_p}\frac{\displaystyle\int_0^{T_p} \inner{\nabla  G(u)}{u}dt}{\displaystyle\int_0^{T_p} G^\star(\nabla G(u))dt}\end{equation*}
is independent of the solution.
Consequently, we can take the solution of $J\dot{u}=-\nabla   G(u)$ satisfying $ G(u(0))=1$. Since  $p$-Laplacian equation \eqref{eq:p.laplacian} has gives rise to an autonomous Hamiltonian system (with Hamiltonian function $- G$), we have that $G(u(t))\equiv 1$ for every  $t\in [0,1]$.

Let $C$ be the closed simple curve parametrized by $u(t)=(u_1(t),u_2(t))$ and let $D$ be the region inside $C$ whose area is denoted by $A(D)$. Note that $C$ is traveled in clockwise direction. From Green's Theorem
\begin{equation*}
 \int_0^{T_p}\inner{u}{\nabla  G(u)}dt
 = \int_0^{T_p}\inner{u}{-J\dot{u}}dt
  = \int_0^{T_p}\inner{Ju}{\dot{u}}dt=\oint_Cu_2du_1-u_1du_2
  =2\iint_DdA=2A(D).
  \end{equation*}
  Using that the curve $C$ is given implicitly  by the equation $ G(u(t))=1$ and performing the change of variable $r=1-s^p/p$, we have that
  \begin{equation*}
  A(D)=4q^{1/q}\int_0^{p^{1/p}}\left(1-\frac{s^p}{p}\right)^{1/q}ds=4(p-1)^{-1/q}\int_0^1r^{1/q}(1-r)^{-1/q}dr=4(p-1)^{-1/q}B\left(\frac{1}{q}+1,\frac{1}{p}\right)=T_p.
 \end{equation*}
On the other hand, using Young's identity
\begin{equation*}
\int_0^{T_p} G^\star(\nabla G(u))dt=\int_0^{T_p}\inner{u}{\nabla  G(u)}dt-\int_0^{T_p}G(u)dt=T_p.
\end{equation*}
Collecting all computations, we obtain the result of the theorem.
\end{proof}

\begin{remark} In the case $n>1$,  the vector $p$-Laplacian equation \eqref{eq:p.laplacian} was studied in several articles (see \cite{manasevich1999spectrum} for a survey on the subject).  If we write $u_1=(u_{1,1},0,\ldots,0)$ being $u_{1,1}:\R\to\R$ a periodic solution of the scalar $p$-Laplacian equation  \eqref{eq:p.laplacian}, we obtain a solution of the vector $p$-Laplacian equation. This simple observation shows that  $C_ G\geq  p\sin\left(\pi/p\right)/2\pi(p-1)^{1/p}$. However, as it is pointed out in  \cite{manasevich1999spectrum}, the vector $p$-Laplacian equation has other periodic solutions with periods incommensurable with $T_p$. More precisely, the following function
\begin{equation*}u_1(t)=u_0\cos t +v_0\sin t,\end{equation*}
where $u_0,v_0\in\R^{n}$ are fixed vectors with $\inner{u_0}{v_0}=0$ and $|u_0|=|v_0|$,
is solution. These functions satisfy that  $|u_1(t)|:=a$  is constant and $T_{u_1}=2\pi$. Recalling that $u_2=|u_1|^{p-2}u_1$, we have
   \begin{equation*}\frac{\displaystyle \int_0^{2\pi} \inner{\nabla  G(u)}{u}\,dt}{\displaystyle\int_0^{2\pi} G^\star(\nabla G(u))\,dt}=
   \frac{\displaystyle\int_0^{2\pi} |u_1|^p+|u_2|^q\,dt}{\displaystyle\int_0^{2\pi}\frac{|u_1|^p}{q}+\frac{|u_2|^q}{p}\,dt}=2.
   \end{equation*}
Consequently  $C_{ G}\geq 1/\pi$, but it is not a new result because  $p\sin\left(\pi/p\right)/2\pi(p-1)^{1/p}\geq 1/\pi$.

It is asked in \cite{manasevich1999spectrum} if the previous ones are essentially all periodic solutions of the vector $p$-Laplacian equations. As far as we know, this question remains an open question.
\end{remark}


\section{Differentiability of Hamiltonian dual action}
\label{sec:differentiability}

In this section, we establish the differentiability of the dual action.

\begin{theorem}\label{thm:DiffDualAct}
Suppose that $ G\colon\R^{2n}\to [0,+\infty)$ is a differentiable G-function with $ G^\star$ semi-symplectic. Additionally, we assume that
\begin{enumerate}[label=$\arabic*)$]
\item \label{it:h1-prop-H}
$\Hcal:[0,T]\times\R^{2n}\to \R$ is measurable in $t$, continuously differentiable with respect to $u$ and such that $\Hcal(t,\cdot)\in  \Gamma(\R^{2n})$.
\item \label{it:h2-cotaH-conjphi}there exist $\beta, \gamma \in \Lpspace[1]([0,T],\R)$, $\Lambda>\lambda>0$ such that
\begin{equation}\label{eq:cota-H-phi-conj}
 G\left(\lambda u\right)-\beta(t)\leq \Hcal(t,u) \leq  G\left(\Lambda u\right)+\gamma(t)
\end{equation}
\end{enumerate}
Then,  the dual action
\begin{equation}\label{eq:DualAct}
 \chi(v)=\int_0^T \frac{1}{2} \langle J\dot{v},v\rangle+\Hcalast (t,\dot{v})\, dt
\end{equation}
is G\^ateaux differentiable on $\WLGTspace[ G^\star]([0,T],\R^{2n}) \cap \{v|d(\dot{v},\Lpspace[\infty])<\lambda\}$.

Moreover, if $v$ is a critical point of $\chi$ with $d(\dot{v}, \Lpspace[\infty]) < \lambda$, then the function defined by $u=\nabla  \Hcalast(t,\dot{v})$ belongs to $\WLGspace([0,T],\R^{2n})$, solves
\begin{equation*}
\left\{\begin{array} {lll}
\dot{u}&=&J\nabla \Hcal (t,u)
\\
u(0)&=&u(T),
\end{array}
\right.
\end{equation*}
and the relation $\dot{u}=J\dot{v}$ holds.
\end{theorem}

\begin{proof}
First, we conjugate \eqref{eq:cota-H-phi-conj} and we obtain
\begin{equation}\label{eq:cotaH*-phi}
-\gamma(t)\leq G^\star\left(\frac{v}{\Lambda}\right)-\gamma(t)\leq \Hcalast(t,v)\leq G^\star\left(\frac{v}{\lambda}\right)+\beta(t).
\end{equation}
Assumption \ref{it:h1-prop-H} guarantees that $\Hcalast$  is continuously differentiable with respect to $v$. Applying Proposition \ref{prop: cota-conj-phi} to $\Hcalast$ and $ G^\star(v/\lambda )$ instead of $\Hcal$ and $G$, for any $r>1$ we get
\begin{equation}\label{eq:main-phiconj-phi}
G\left(\lambda\nabla \Hcalast(t,v)\right)\leq \frac{1}{r-1}  G^\star\left(r\frac{v}{\lambda }\right)+\frac{r}{r-1}(\beta+\gamma).
\end{equation}

Consider the Lagrangian function $\Lcal:[0,T]\times\R^{2n}\times \R^{2n}\to \R$ given by
\begin{equation}
\Lcal(t,v,\xi)=\frac{1}{2}\langle J\xi,v\rangle + \Hcalast(t,\xi).
\end{equation}
In \cite[Theorem 4.5]{MA2017}, it was proved that if  there exist $\Lambda_0,\lambda_0>0$ and functions $a \in C(\R^{2n},\R)$
and $b\in \Lpspace[1]([0,T],\R)$ such that
\begin{equation}\label{eq:cota-L-gradL-conj}
|\mathcal{L}|+|\nabla_v\mathcal{L}|+ G\left(\frac{\nabla_{\xi}\mathcal{L}}{\lambda_0}\right)\leq
a(v)\left(b(t)+ G^\star\left(\frac{\xi}{\Lambda_0}\right)\right),
\end{equation}
then $\chi$, which is the action functional corresponding to $\mathcal{L}$, is G\^ateaux differentiable on the set
$\WLGTspace[ G^\star]([0,T],\R^{2n})\cap\{v|d(\dot{v},\Lpspace[\infty]) < \Lambda_0\}$.

In order to show that an inequality like \eqref{eq:cota-L-gradL-conj} holds, first we provide an estimation for
$\mathcal{L}$.
From \eqref{eq:cotaH*-phi} and since $J$ is orthogonal, we have
\begin{equation*}
|\mathcal{L}|\leq  \frac{1}{2}|\xi||v|+ G^\star\left(\frac{\xi}{\lambda}\right)+\beta(t).
\end{equation*}
Since $\frac{ G^\star(v)}{|v|}\to \infty$ as $|v|\to \infty$, there exists $C>0$ such that $|v|\leq  G^\star(v)+C$ for all $v \in \R^{2n}$.
Then,
\begin{equation}\label{eq:cota_lx}
|\mathcal{L}|\leq \frac{1}{2} \lambda|v|\left[ G^\star\left(\frac{\xi}{\lambda}\right)+C\right]
+ G^\star\left(\frac{\xi}{\lambda}\right)+\beta(t)
\leq
\max\{ 1,\lambda|v|\}
\left[ G^\star\left(\frac{\xi}{\lambda}\right)+C+\beta(t)\right],
\end{equation}
which is an estimate like  the right hand side of \eqref{eq:cota-L-gradL-conj}.

Now, we provide an estimate for $|\nabla_v \mathcal{L}|$. Applying the same technique as above, we get
\begin{equation}\label{eq:cota-grad-v}
|\nabla_v \mathcal{L}|=\frac{1}{2}|J\xi|\leq |\xi| \leq \lambda\left[ G^\star\left(\frac{\xi}{\lambda}\right)+C\right],
\end{equation}
which is also an estimate of the desired type.

Finally, we deal with $ G(\nabla_{\xi}\mathcal{L}/{\lambda_0})$.
Since $ G$ is a convex, even function, we have
\begin{equation*} G\left(\frac{\nabla_{\xi}\mathcal{L}}{\lambda_0}\right)=
 G\left(\frac{-\frac{1}{2}Jv}{\lambda_0}+\frac{\nabla \Hcalast(t,\xi)}{\lambda_0}\right)
\leq
\frac12  G\left(\frac{Jv}{\lambda_0}\right)+
\frac12 G\left(\frac{2\nabla \Hcal^{\star}(t,\xi)}{\lambda_0}\right).
\end{equation*}
Now, choosing $\lambda_0=2/\lambda$ and applying \eqref{eq:main-phiconj-phi}, we  have
\begin{equation}\label{eq:cota-conj-grad}
\begin{split}
 G\left(\frac{\nabla_{\xi}\mathcal{L}}{\lambda_0}\right)&\leq
\frac12 G\left(\frac{\lambda Jv}{2}\right)+\frac12 G^\star\left(r\frac{\xi}{\lambda }\right)+\frac12\frac{r}{r-1}(\beta+\gamma)\\&=
\frac12\max\left\{ G\left(\frac{\lambda Jv}{2}\right),1\right\}
\left[ G^\star\left(r\frac{\xi}{\lambda}\right)+\frac{r}{r-1}(\beta+\gamma)\right],
\end{split}
\end{equation}
which again is an estimate of the desired form.

Therefore, from \eqref{eq:cota-grad-v},\eqref{eq:cota_lx} and \eqref{eq:cota-conj-grad},
we see that condition \eqref{eq:cota-L-gradL-conj} holds
for appropriate functions $a$ and $b$  and for $\Lambda_0=\lambda/r$.

This implies differentiability of $\chi$ in a set $\WLGTspace[ G^\star]([0,T],\R^N)\cap\{v|d(\dot{v},\Lpspace[\infty])<\Lambda_0\}$. Since $r$ is any number bigger than 1, $\Lambda_0$ is arbitrary close to $\lambda$. Thus $\chi$ is differentiable on $\WLGTspace[ G^\star]([0,T],\R^N)\cap\{v|d(\dot{v},\Lpspace[\infty])<\lambda\}$.

Let  $v \in \WLGTspace[ G^\star]([0,T],\R^N)\cap \{d(\dot{v},\Lpspace[\infty])<\lambda\}$
be a critical point of $\chi$. Then, from \cite[Theorem 4.5]{MA2017}, we obtain
\begin{equation*}
\int_0^T\langle \nabla \Hcalast(t,\dot{v})-\frac12J v,\dot{h} \rangle dt=
-\int_0^T \frac{1}{2} \langle J\dot{v},h\rangle dt.
\end{equation*}

From Proposition \ref{prop:embbeding_for_symplectic_G} and \eqref{eq:main-phiconj-phi} we deduce that the functions $\nabla\Hcalast(t,\dot{v})-\frac12Jv$ and  $J\dot{v} $ are in the space $\LGspace$. Since  $\LGspace\hookrightarrow \Lpspace[1]$, from the Fundamental Lemma (see \cite[Chapter 1]{mawhin2010critical})  we deduce that  $\Hcalast(t,\dot{v})-\frac12J v$ is absolutely continuous. It follows that $v$ solves $J\dot{v}=\nabla \Hcalast(t,\dot{v})$ and therefore by duality we obtain desired result.
\end{proof}

\begin{remark}
If in addition we assume that $G^\star\in \Delta_2$ then $d(\dot{v},\Lpspace[\infty])=0$, since $\Lpspace[\infty]$ is dense in $\EGspace[G^\star]=\LGspace[G^\star]$. In this case $\chi$ is continuously differentiable on the whole space $\WLGTspace[G^\star]$ (see \cite{MA2017}).
\end{remark}


\section{Existence of  periodic solutions for Hamiltonian system}
\label{sec:existence}

The following theorem establishes the existence of minimum for the dual action functional. Our result is a generalization of \cite[Theorem 3.1]{mawhin2010critical}, where the existence was established for $|u|^2/2$. Even for the function $|u|^2/2$ our theorem is slightly better than \cite[Theorem 3.1]{mawhin2010critical}. We obtain existence when under assumption that the functions $\xi\Lpspace[2]$ and $\alpha\Lpspace[1]$ instead of $\Lpspace[4]$ and $\Lpspace[2]$ respectively, as it was assumed in \cite{mawhin2010critical}. This little improvement is due to the observation in Remark \ref{com:mejora_desi}.

\begin{theorem}\label{thm:solution-ham}
Suppose that $ G\colon \R^{2n}\to [0,\infty)$ is a G-function such that  $G\in\Gamma(\R^{2n})$, $G^\star$ is semi-symplectic and $G^\star \in \Delta_2$.  Assume  $\Hcal\colon [0,T]\times\R^{2n}\to \R$ is $C^1$ and $\Hcal(t,\cdot)\in  \Gamma(\R^{2n})$. Additionally suppose that

\begin{enumerate}[label=$(H_\arabic*)$]

\item\label{it:hip1}
There exists $\xi\in \LGspace[ G^\star]([0,T],\R^{2n})$ such that for every $u \in \R^{2n}$ and a.e. $t \in [0,T]$
\begin{equation*}
\Hcal(t,u)\geq \langle \xi(t), u\rangle.
\end{equation*}

\item\label{it:hip2}
There exist $\Lambda $ with $\Lambda^{-1}> T\max\{1,C_{ G^\star}(T)/2\}$ and $\alpha \in \Lpspace[1]([0,T],\R) $
such that, for every $u\in \R^{2n}$ and a.e. $t \in [0,T]$, we have
\begin{equation*}
\Hcal(t,u)\leq  G\left( \Lambda u\right)+\alpha(t).
\end{equation*}

\item\label{it:hip3}
\begin{equation*}
  \int_0^T\Hcal(t,u)dt\to+\infty,\quad\hbox{when } |u|\to+\infty.
\end{equation*}

\end{enumerate}
Then, there exists $u\in\WLGTspace([0,T],\R^{2n})$ which is a solution of the problem
\begin{equation}
    \tag{HS}
    \label{eq:hamiltonian}
\begin{cases}
    \dot{u} = J \nabla \Hcal(t,u),& \text{ a.e. on $[0,T]$}\\
    u(0) = u(T),
\end{cases}
\end{equation}
and such that $v=-J\widetilde{u}$ minimizes the dual action
\begin{equation*}
\chi(v)= \int_0^T \frac{1}{2} \langle J\dot{v},v\rangle+\Hcalast(t,\dot{v})\,dt.
\end{equation*}
\end{theorem}

\begin{proof}
\textbf{Step 1:}
Suppose that $0<r<1$ and $\varepsilon>0$ are small enough to have
\begin{equation*}
\Lambda^{-1}> (1+r)T\max\{1,C_{G^\star}(T)/2\}\quad\hbox{and}\quad\varepsilon<r\Lambda.
\end{equation*}
We define the perturbed Hamiltonian by
\begin{equation*}
  \Hcal_{\varepsilon}(t,u)=\Hcal(t,u)+G(\varepsilon u).
\end{equation*}
By \ref{it:hip1}, inequality \eqref{ineq:generalFenshel} and Proposition \ref{prop:elem_ine}, we have
\begin{equation}\label{eq:CotaSupH}
\Hcal_{\varepsilon}(t,u)
\geq \left\langle \xi(t), u \right\rangle +G(\varepsilon u)
\geq -G^\star\left(\frac{1}{r\varepsilon} \xi(t)\right)-G(r\varepsilon u)+G(\varepsilon u)
\geq  G((1-r)\varepsilon u)-\beta(t),
\end{equation}
where, since $G^\star\in\Delta_2$, $\beta(t):=G^\star( \frac{1}{r\varepsilon} \xi(t))\in \Lpspace[1]$.
On the other hand,  Proposition \ref{prop:elem_ine} implies that
\begin{equation}\label{eq:CotaInfH}
  \Hcal_{\varepsilon}(t,u)  \leq   G\left(\Lambda u\right)+\alpha(t)+ G(\varepsilon u)\leq
     G\left((1+r)\Lambda u\right)+\alpha(t).
\end{equation}
From  \eqref{eq:CotaSupH}, \eqref{eq:CotaInfH} and  properties of Fenchel conjugate, we get
\begin{equation}\label{eq:CotaH*}
 G^{\star}\left(\frac{v}{(1+r)\Lambda}\right)-
\alpha(t)\leq \Hcalast_{\varepsilon}(t,v)\leq  G^{\star}\left(\frac{v}{(1-r)\varepsilon}\right)
+\beta(t).
\end{equation}

Define the perturbed dual action $\chi_\varepsilon\colon \WLGTspace[G^\star]([0,T],\R^{2n}) \to \R$ by
\begin{equation}\label{eq:DualActDelta}
\chi_{\varepsilon}(v)=\int_0^T \frac{1}{2} \langle J\dot{v},v\rangle
+\Hcal_{\varepsilon}^{\star}(t,\dot{v})\,dt.
\end{equation}
From \eqref{eq:CotaH*} and \eqref{eq:cotaJu}, we have
\begin{equation*}
\chi_{\varepsilon}(v)\geq
-\frac{C_{G^{\star}}(T)}{2} \int_0^T G^{\star} (T\dot{v})\,dt+\int_0^T G^{\star}
\left(\frac{ \dot{v}}{(1+r)\Lambda}\right)\,dt-\int_0^T \alpha(t)\,dt-C_2.
\end{equation*}
Thus, as $T(1+r)\Lambda<1$ we obtain
\begin{equation}\label{eq:bound_coercivity}
  \begin{split}
      \chi_{\varepsilon}(v)&\geq
      -\frac{C_{ G^{\star}}(T)}{2} \int_0^T  G^{\star} (T\dot{v})\,dt+\frac{1}{T(1+r)\Lambda}\int_0^T  G^{\star}
      \left(T\dot{v}\right)\,dt-\int_0^T \alpha(t)\,dt-C_2
      \\
      &>\left(\frac{1}{T(1+r)\Lambda}- \frac{C_{ G^{\star}}(T)}{2}   \right)\int_0^T  G^{\star} (T\dot{v})\,dt-\int_0^T \alpha(t)\,dt-C_2\\
      &=:C_{\chi}\int_0^T  G^{\star} (T\dot{v})\,dt-B_{\chi}.\\
  \end{split}
\end{equation}

By the definition of $\Lambda$ and our choice of $r$ we have that $C_{\chi}>0$. Since $\chi_{\varepsilon}(v)=\chi_{\varepsilon}(v+c)$ with $c \in \R^{2n}$, it is sufficient to minimize $\chi_\varepsilon$ on $\WTLGtildespace[G^\star]([0,T],\R^{2n})$.

The perturbed dual action is coercive on this space. To see this let $\{v_n\}\subset\WTLGtildespace[ G^\star]([0,T],\R^{2n})$ and suppose that $\WLGnorm[ G^\star]{v_n} \to \infty$. Then $\LGnorm[ G^\star]{\dot{v}_n} \to \infty$ or $|\overline{v}_n|\to \infty$. Since $\overline{v}_n=0$, $\LGnorm[ G^\star]{\dot{v}_n} \to \infty$. Hence from \eqref{eq:lowbound_for_modular_by_norm2} we obtain that $\int_0^T G^{\star} (T\dot{v}_n)\,dt \to \infty$ and consequently $\chi_{\varepsilon}(\dot{v}_n)\to\infty$, by \eqref{eq:bound_coercivity}.

It follows that if $\{v_n\}\subset\WTLGtildespace[ G^\star]([0,T],\R^{2n})$ is a minimizing sequence for $\chi_{\varepsilon}$ then $\dot{v}_n$ is a bounded sequence in $\LGspace[ G^\star] = \left(\EGspace[ G]\right)^\star$. Following a standard argument (see \cite[Theorem 3.2]{MA2017}),
we obtain a function $v_{\varepsilon}\in\WTLGtildespace[ G^\star]([0,T],\R^{2n})$ which is a minimum
of $\chi_{\varepsilon}$.

As $ G^\star \in\Delta_2$ then $\Lpspace[\infty]$ is dense in $\LGspace[ G^\star]$ and consequently $d(\dot{v}_{\varepsilon},\Lpspace[\infty])=0$. Theorem \ref{thm:DiffDualAct} implies that
\begin{equation*}
u_\varepsilon(t) = \nabla \Hcal_\varepsilon^\star(t,\dot{v}_\varepsilon) \in \WLGTspace([0,T],\R^{2n})
\end{equation*}
is a solution to
\begin{equation}
\begin{cases}
\label{eq:hamiltonian:existence:perturbedproblem}
\dot{u}
=J\nabla \Hcal_\varepsilon(t,z)=\varepsilon
J  \nabla  G\left(\varepsilon u \right)
+
J\nabla\Hcal(t,u)
\\
u(0)=u(T)
\end{cases}
\end{equation}
and the relation $\dot{u}_\varepsilon=J\dot{v}_\varepsilon$ holds.

\noindent\textbf{Step 2:}  Now, we provide a posteriori estimates on $u_\varepsilon = \nabla \Hcalast_\varepsilon(t,v_\varepsilon)$.
It is easy to verify (see \cite[page 47]{mawhin2010critical}) that there exists $\overline{u}\in\R^{2n}$ such that
\begin{equation*}
\int_0^T \nabla \Hcal(t,\overline{u})\, dt=0.
\end{equation*}
We define
\begin{equation*}
w(t)=\int_0^t \nabla \Hcal(s,\overline{u})\, ds+c,
\end{equation*}
where $c$ is chosen in order to  $ \int_0^T w\, dt = 0$.
    The function $w$ is absolutely continuous, we show that $w\in \WLGspace$.
From \ref{it:hip1}, \ref{it:hip2} and inequality \eqref{ineq:generalFenshel}, it follows  that for any $t\in[0,T]$ and $u\in\R^{2n}$
\begin{equation*}-G^\star\left(\frac{\xi(t)}{\Lambda}\right) \leq
    \Hcal(t,u)+G(\Lambda u)\leq 2G(\Lambda u)+\alpha(t).
\end{equation*}
Therefore, the function $\Hcal(t,u)+G(\Lambda u)$ and the G-function $2G(\Lambda u)$ satisfy hypothesis of Proposition \ref{prop: cota-conj-phi}.
Consequently, taking $r=2$
\begin{equation*}
  G^\star\left(\frac{\nabla\Hcal(t,u)+\Lambda\nabla G(\Lambda u)}{2\Lambda}\right)\leq  G(2\Lambda u)+ 2 G^\star\left(\frac{\xi(t)}{\Lambda}\right)+ 2\alpha(t).
\end{equation*}
This inequality and the fact that $\dot{w}= \nabla\Hcal(t,\overline{u})$ imply that
\begin{equation}\label{eq:bound-G*-L1}
\begin{split}
   G^\star\left( \frac{\dot{w}}{4\Lambda} \right)
  &=
   G^\star\left(\frac{\nabla\Hcal(t,\overline{u})}{4\Lambda} \right)\\
  &\leq
  \frac12 G^\star\left(\frac{\nabla\Hcal(t,\overline{u})+\Lambda\nabla G(\Lambda \overline{u})}{2\Lambda}\right)+\frac12 G^\star\left(\nabla G\left(\frac{\Lambda \overline{u}}{2}\right)\right)\\
  &\leq 2 G(2\Lambda \overline{u})+ 2 G^\star\left(\frac{\xi(t)}{\Lambda}\right)+ 2\alpha(t)
  \in \Lpspace[1].
\end{split}
\end{equation}
Thus $\dot{w}\in\LGspace[ G^\star]$. Moreover, $\Hcal^\star(t,\dot{w})=\inner{\dot{w}}{\overline{u}}-\Hcal(t,\overline{u})$ so that $\Hcal^\star(\cdot,\dot{w}(\cdot))\in \Lpspace[1]([0,T],\R)$.


From inequality $\Hcal(t,u)\leq \Hcal_\varepsilon(t,u)$, we deduce that $\Hcal_\varepsilon^\star(t,v)\leq
\Hcal^\star(t,v)$. By  inequality \eqref{eq:bound_coercivity} and \eqref{eq:bound-G*-L1}, we have
\begin{equation*}
    C_{\chi}\int_0^T  G^{\star} (T \dot{v}_\varepsilon)\,dt-B_{\chi}
     \leq \chi_\varepsilon(v_\varepsilon)
    \leq \chi_\varepsilon(w)
    \leq  \int_0^T \frac12\inner{J\dot{w}}{w} + \Hcal^\star(t,\dot{w})\,dt
    =:c_1 < \infty.
\end{equation*}
Since $ G^\star$ is semi-symplectic, there exist $C,k>0$ with $ G(Ju)\leq  G^\star(ku)+C$. Moreover, since $\dot{u}_\varepsilon=J\dot{v}_\varepsilon$, we have
\begin{equation*}
\int_0^T  G\left( \frac{ T}{k} \dot{u}_\varepsilon \right) \,dt
=
\int_0^T  G\left( \frac{ T}{k} J\dot{v}_\varepsilon \right) \,dt
\leq
C+\int_0^T  G^\star\left( T\dot{v}_\varepsilon \right) \,dt \leq c_2,
\end{equation*}
and
\begin{equation}\label{eq:ue-ve-relation}
  Jv_\varepsilon=u_\varepsilon-\overline{u}_\varepsilon.
\end{equation}
It follows from  \eqref{eq:lowbound_for_modular_by_norm2} that $\dot{u}_\varepsilon$ is uniformly bounded in $\LGspace$. Now, from inequality \eqref{ineq:conjugateP-W} we deduce that  $\widetilde{u}_\varepsilon$ is uniformly bounded in $\Lpspace[\infty]$. Therefore, there exists $c_3$ such that
\begin{equation*}
\int_0^T G\left(\Lambda \widetilde{u}_\varepsilon\right)dt\leq c_3.
\end{equation*}
Thus, using \eqref{eq:ue-ve-relation} and Theorem \ref{thm:LowBoundQuadraticForm}, we have \begin{equation}\label{eq:lower-bound-omega-eps}
\int_0^T\inner{J\dot{u}_\varepsilon}{u_\varepsilon}dt= \int_0^T\inner{-\dot{v}_\varepsilon}{Jv_\varepsilon+\overline{u}_\varepsilon}dt\geq -C_{ G^\star} \int_0^T G^\star\left( T\dot{v}_\varepsilon \right)-C_1\geq -c_4.
\end{equation}

The convexity of $\mathcal{H}(t,\cdot)$, inequality \eqref{ineq:G_by_gradientG},
\ref{it:hip2}, \eqref{eq:hamiltonian:existence:perturbedproblem} and
the fact that $\inner{u}{\nabla G(u)}\geq 0$ for any $u\in\R^{2n}$, imply
\begin{equation*}
  \begin{split}
      2\mathcal{H}\left(t,\frac{\overline{u}_{\varepsilon}(t)}{2}\right)
      &\leq
      \mathcal{H}(t,u_{\varepsilon})
      +
      \mathcal{H}(t,-\widetilde{u}_{\varepsilon})
       \\
      &\leq
      \inner{\nabla \mathcal{H}(t,u_{\varepsilon}(t))}{u_{\varepsilon}}
      +
      \mathcal{H}(t,0)+ G(\Lambda \widetilde{u}_\varepsilon)+\alpha(t)\\
       &=
      \inner{-J\dot{u}_\varepsilon-\varepsilon\nabla G(\varepsilon u_\varepsilon)}{u_{\varepsilon}}
      +
      \mathcal{H}(t,0)+ G(\Lambda \widetilde{u}_\varepsilon)+\alpha(t)\\
      &\leq
      \inner{-J\dot{u}_\varepsilon}{u_{\varepsilon}}
      +
      \mathcal{H}(t,0)+ G(\Lambda \widetilde{u}_\varepsilon)+\alpha(t).\\
    \end{split}
\end{equation*}
Integrating the previous inequality and using \eqref{eq:lower-bound-omega-eps}, it follows that
\begin{equation*}
\int_0^T \mathcal{H}\left(t,\frac{\overline{u}_{\varepsilon}}{2}\right) \,dt \leq c_5.
\end{equation*}
Now, by \ref{it:hip3} we have that $\overline{u}_{\varepsilon}$ is uniformly bounded. Thus, we have that $u_\varepsilon$ is uniformly bounded in $\WLGspace([0,T],\R^{2n})$.

\noindent \textbf{Step 3:} By a  standard  argument (see \cite{MA2017}), we can suppose that there exists a sequence $\varepsilon_n$ such that $u_n:=u_{\varepsilon_n}$ converges uniformly to a continuous function $u\in \WLGspace([0,T],\R^N)$ and  that $\dot{u}_n$ converges to $\dot{u}$ in the weak$^\star$ topology of $\LGspace([0,T],\R^N)$.
From \eqref{eq:hamiltonian:existence:perturbedproblem} in integrated form
\begin{equation*}
J u_{n}(t)-J u_{n}(0)=
- \int_0^t
\varepsilon_n \nabla G(\varepsilon_n u_{n})
+\nabla \Hcal(t,u_{n})
\,dt,
\end{equation*}
we deduce $u$ is a solution of the original problem.

It remains to prove that $v$ minimizes the dual action integral. Since $\dot{v}_n=\nabla \Hcal(t,u_n)$, we have
\begin{equation*}
  \begin{split}
      \chi_{\varepsilon_n}\left(v_{\varepsilon_n}\right) &=
       \int_0^T
            \left[\frac12\inner{J\dot{v}_n}{v_n}+\inner{u_n}{\dot{v}_n}-\Hcal_{\varepsilon_n}(t,u_n) \right]\\
      &=\int_0^T\left[\frac12\inner{J\dot{v}_n}{v_n}+\inner{u_n}{\dot{v}_n}-\Hcal(t,u_n)-G\left(\varepsilon_n u_n\right) \right]\,dt.\\
   \end{split}
\end{equation*}
Taking into account that   $\dot{v}_n\overset{\star}{\rightharpoonup} \dot{v}$ in $\LGspace[ G^\star]$ and $u_n \to u$ uniformly, we obtain
\begin{equation*}
  \begin{split}
    \lim_{n\to\infty} \chi_{\varepsilon_n}\left(v_{\varepsilon_n}\right) &=
        \lim_{n\to\infty} \int_0^T\left[\frac12\inner{J\dot{v}_n}{v_n}+\inner{u_n}{\dot{v}_n}-\Hcal(t,u_n)-G\left(\varepsilon_n u_n\right) \right]\,dt\\
    &= \int_0^T\left[\frac12\inner{J\dot{v}}{v}+\inner{u}{\dot{v}}-\Hcal(t,u)\right]\,dt.\\
  \end{split}
\end{equation*}
Now,  \eqref{eq:ue-ve-relation} implies that $v=-J(u-\overline{u})$.
Thus, using  \eqref{eq:hamiltonian} we get $\dot{v}=-J\dot{u}=\nabla\Hcal(t,u)$. Consequently,
\begin{equation*}
  \begin{split}
    \lim_{n\to\infty} \chi_{\varepsilon_n}\left(v_{\varepsilon_n}\right) &=
        \int_0^T\left[\frac12\inner{J\dot{v}}{v}+\Hcal^\star(t,\dot{v})\right]dt=\chi(v).
  \end{split}
\end{equation*}
On the other hand, from $\Hcal_\varepsilon^\star\leq \Hcalast$ we  have that for any $w\in\WLGTspace[ G^\star]([0,T],\R^{2n})$, $\chi_{\varepsilon_n}\left(v_{\varepsilon_n}\right)\leq \chi_{\varepsilon_n}\left(w\right) \leq  \chi\left(w\right)$. Therefore, $v$ is  a minimum of $\chi$.
\end{proof}

In the case where $G(x)=|x|^2/2$, in  \cite[Theorem 3.1]{mawhin2010critical}  it is assumed that constant $\Lambda<\sqrt{2\pi/T}$. Meanwhile, in \ref{it:hip2} we are assuming that $\Lambda<\min\{1/T,2\pi\}$, i.e. when  $G(x)=|x|^2/2$ our constant $\Lambda$ is not as good as constant in  \cite[Theorem 3.1]{mawhin2010critical}.  Assuming additional hypothesis on the G-function $G$, we are able to obtain better estimates for the constant $\Lambda$.

First, we recall some definitions from \cite[Chapter 11]{Maligranda}. In that monograph, it were considered  a G-function such that $ G\colon\R\to [0,+\infty)$. However, all definitions and results remains true in the anisotropic setting.

We denote by $\alpha_{ G}$ and $\beta_{ G}$ the so-called  \emph{Matuszewska-Orlicz indices} of the function $ G$, which are defined by
\begin{equation}\label{MO_indices}
    \alpha_{ G}:=\lim\limits_{t\to 0^{+}}\frac{\log \left (\sup\limits_{u\neq 0}\frac{ G(t u)}{ G(u)} \right ) }{\log(t)},\quad
    \beta_{ G}:=\lim\limits_{t\to +\infty}\frac{\log \left  (\sup\limits_{u\neq 0}\frac{ G(t u)}{ G(u)}\right )}{\log(t)}.
\end{equation}
We have that $0\leq \alpha_{ G}\leq \beta_{ G}\leq +\infty$.
The relation $\beta_{ G}<\infty$ holds if and only if $ G$ satisfies the $\Delta_2$-condition globally. On the other hand, $\alpha_{ G}>1$ if and only if $ G^\star$ satisfies the $\Delta_2$-condition  globally.

In the case that  $ G$  and $ G^\star$ satisfy the $\Delta_2$-condition globally,  for every $\epsilon>0$ there exists a
constant $K=K( G,\epsilon)$ such that, for every $t,u\geq 0$,
\begin{equation}\label{delta2-potencias}
    K_{ G,\varepsilon}^{-1}\min\big\{t^{\beta_{ G}+\epsilon},t^{\alpha_{ G}-\epsilon} \big\} G(u)\leq  G(t u)\leq
    K_{ G,\varepsilon}\max\big\{t^{\beta_{ G}+\epsilon},t^{\alpha_{ G}-\epsilon} \big\} G(u).
\end{equation}

\begin{proposition}\label{prop:optimal_lambda} The conclusions of Theorem
\ref{thm:solution-ham} continue to be true if  we suppose that $ G$ and $ G^\star$ satisfy the  $\Delta_2$-condition globally and instead of inequality  $\Lambda^{-1}> T\max\{1,C_{ G^\star}(T)/2\}$ in \ref{it:hip2}, we  assume that
\begin{equation}\label{eq:lambda_opt}
  K_{ G,\varepsilon}^{-1}\min\big\{(T\Lambda)^{-\beta_{ G}-\epsilon},(T\Lambda)^{-\alpha_{ G}+\epsilon} \big\}\geq \frac{C_{ G^\star}(T)}{2},
\end{equation}
 where the constant $K_{ G,\varepsilon}, \alpha_{ G}, \beta_{ G}$ satisfy \eqref{delta2-potencias}.
\end{proposition}

\begin{proof} The only change that must be made in the proof of Theorem  \ref{thm:solution-ham} is choosing  $0<r<1$ such that
 \begin{equation*}
  K_{ G,\varepsilon}^{-1}\min\left\{[(1+r)T\Lambda]^{-\beta_{ G}-\epsilon},[(1+r)T\Lambda]^{-\alpha_{ G}+\epsilon} \right\}\geq \frac{C_{ G^\star}(T)}{2}.
\end{equation*}
Now, we use \eqref{eq:lambda_opt} to produce the next inequality
\begin{equation*}
\int_0^T  G^{\star}
    \left(\frac{ \dot{v}}{(1+r)\Lambda}\right)\,dt\geq K_{ G,\varepsilon}^{-1}\min\left\{[(1+r)T\Lambda]^{-\beta_{ G}-\epsilon},[(1+r)T\Lambda]^{-\alpha_{ G}+\epsilon} \right\}\int_0^T  G^{\star}
    \left(T\dot{v}\right)\,dt.
\end{equation*}
From here, the proof continues like in Theorem \ref{thm:solution-ham}.
\end{proof}

\begin{remark}
If $ G_2(u)=|u|^2/2$ then inequality \eqref{delta2-potencias} holds with  $\epsilon=0$, $K_{ G_2,\varepsilon}=1$, $\alpha_{ G_2}=\beta_{ G_2}=2$.
Since $ G_2^\star=G_2$, from \eqref{eq:CG2_constant} we have that $C_{ G_2^\star}=1/T\pi$.
Thus inequality \eqref{eq:lambda_opt} is equivalent to $\Lambda\leq \sqrt{2\pi/T}$, which is the same constant as in \cite[Theorem 3.1]{mawhin2010critical}. On the other hand, in Theorem \ref{thm:solution-ham} we assume that  $\xi\in \LGspace[2]$ and $\alpha\in \LGspace[1]$.  Meanwhile in order to apply \cite[Theorem 3.1]{mawhin2010critical} we need  $\xi\in \LGspace[4]$ and $\alpha\in \LGspace[2]$.
\end{remark}

\begin{remark}
\label{rem:comp_tian}
Let us discuss the relation between Proposition \ref{prop:optimal_lambda} and the result obtained in  \cite{TiaGe07}. Recall that they consider Hamiltonian $\Hcal$ given by \eqref{eq:hamil_tian} and satisfying \eqref{eq:A1} and \eqref{eq:A2}.

If $\xi(t)=(l_1(t),l_2(t))$ is a function satisfying
\begin{equation*}
 \inner{\xi(t)}{u}\leq \frac{1}{a}F(t,u_1)+\frac{a^{q-1}}{q}|u_2|^q,
\end{equation*}
for any $u\in\R^{2n}$, then taking $u=(0,u_2)$ we have that $\inner{l_2(t)}{u_2}\leq\frac{a^{q-1}}{q}|u_2|^q$ and this inequality is true only for $l_2\equiv 0$. Consequently \ref{it:hip1} implies
\begin{equation*}
 \inner{l_1(t)}{u_1}\leq \frac{1}{a}F(t,u_1),
\end{equation*}
and $l_1\in \LGspace[q]$ (recall that $G^\star_p(l_1,l_2)=|l_1|^q/q+|l_2|^p/p$). Therefore our condition \ref{it:hip1} differs slightly from \eqref{eq:A1}.

The condition \eqref{eq:A2} for $\Hcal$  implies
\begin{equation*}
\Hcal(t,u_1,u_2)\leq \frac{a}{p}|u_1|^p+\frac{a^{q-1}}{q}|u_2|^{q}+\frac{\gamma(t)}{a}=G_p(\Lambda u_1,\Lambda u_2)+\alpha(t),
\end{equation*}
where $G_p(u_1,u_2)=|u_1|^p/p+|u_2|^q/q$, $\Lambda=a^{1/p}$ and $\alpha(t)=\gamma(t)/a$. The  inequality $0<a <\min\{T^{-\frac{p}{q}}, T^{-1}\}$ shows that condition \eqref{eq:A2} in \cite{TiaGe07} implies
\begin{equation*}
\Lambda<\min\{T^{-\frac{1}{q}}, T^{-\frac{1}{p}}\}.
\end{equation*}
On the other hand, it is easy to see that inequality \eqref{delta2-potencias} holds with $K_{ G_p,\varepsilon}=1$, $\alpha_G=\min\{p,q\}$, $\beta_G=\max\{p,q\}$ and $\varepsilon=0$.  Therefore, the fact that $C_{G^\star}(T)=C_{G^\star}(1)/T=C_{G^\star}/T$ implies that $\Lambda$ satisfies inequality \eqref{eq:lambda_opt} if and only if
\begin{equation*}
\Lambda < \min\left\{\left(\frac{2}{C_{G^\star}}\right)^{1/p} T^{-1/q},\left(\frac{2}{C_{G^\star}}\right)^{1/q} T^{-1/p}\right\}.                                                                                                                                                                                              \end{equation*}
Recalling that for $G$ symplectic we have $C_{G^\star}\leq 2$, we obtain that the condition  \eqref{eq:A2} implies our condition \ref{it:hip2}. We suspect that the estimate $C_{G^\star}\leq 2$ is not the best possible (it is evident when $n=1$).
\end{remark}

\begin{remark}
\label{rem:condH1'}
To finish this section let us give a condition that contains conditions \eqref{eq:A1} and \ref{it:hip1}  as particular cases. 
Concretely, Theorem \ref{thm:solution-ham} remains true if we replace \ref{it:hip1}  by the following condition.
\begin{description}
 \item[($H1'$)]There exist $b\in \LGspace[1]([0,T],\R)$, a $G$-function $G_0:\R^{2n}\to[0,\infty)$, $\xi\in \EGspace[G^\star_0]$ and a   map $f:\R^{2n}\to\R^{2n}$ such that
 \begin{equation*}
  H(t,u)\geq \inner{\xi(t)}{f(u)}+b(t),\quad u\in\R^{2n},t\in[0,T]
 \end{equation*}
 and $f$ satisfies that for every $\epsilon>0$ there exists $\delta>0$ such that
 \begin{equation*}
  G_0(\delta f(u))\leq G(\epsilon u),\quad u\in\R^{2n}
 \end{equation*}
\end{description}
Note that  condition \eqref{eq:A1} for Hamiltonian \eqref{eq:hamil_tian} is obtained choosing $f(u_1,u_2)=(|u_1|^{\frac{p-2}{2}}u_1,0)$, $\xi(t)=(l(t),0)$, $b\equiv 0$ and $G_0(u)=|u|^2$ in condition ($H1'$) and finally taking $u_2=0$.
Hence we obtain existence when $l\in \Lpspace[2]$, while in \cite{TiaGe07} it is assumed $l\in\Lpspace[2\max\{p,q-1\}]$.
\end{remark}


\section{Application to the existence of solutions of second order systems}
\label{sec:applic}

The purpose of this section is to apply the previous results to get existence of solutions
of the second-order system
\begin{equation}\label{eq:EL}
\begin{cases}
    \frac{d}{dt}\nabla \Phi(\dot{q})+\nabla V(t,q)=0\quad \text{for a.e. }t\in [0,T] \\
    q(0)=q(T)   ,\quad  \dot{q}(0)=\dot{q}(T),
\end{cases}\tag{EL}
\end{equation}
where $\Phi\colon \R^N\to \R$ is a G-function in $\Gamma(\R^N)$ such that $\Phi$ and $ \Phi^\star$ satisfy $\Delta_2$ condition and $V\colon [0,T]\times \R^N\to \R$, $(t,q)\mapsto V(t,q)$ is a  Carath\'eodory function continuously differentiable and convex in $q$.

\begin{theorem}
\label{thm:main}
Assume that the following conditions are satisfied:
\begin{enumerate}[label=($V_\arabic*$),series=V]
\item
\label{asm:V:lowerbound}
there exists $l\in \LGspace[\Phi]([0,T],\R^N)$ such that for all $q\in \R^N$ and
a.e. $t\in[0,T]$, such that
\begin{equation*}
\inner{l(t)}{q}\leq V(t,q);
\end{equation*}
\item
\label{asm:V:upperbound}
for all $x\in \R^N$ and a.e. $t\in[0,T]$ one has
\begin{equation*}
V(t,q)\leq \Phi\left(\Lambda^2 q\right)+\gamma(t);
\end{equation*}
where $\Lambda^{-1}>T\max\{1,C_{ G}/2\}$.

\item
\label{asm:V:coercivity}
\begin{equation*}
\lim_{x\to \infty} \int_0^T V(t,x) \,dt=\infty.
\end{equation*}
\end{enumerate}
where $C_{ G}=C_{ G}(T)$ denotes constant corresponding to the G-function $G(q,p)=\Phi(q)+\Phi^\star(p)$. Then the problem \eqref{eq:EL} has at least one solution.
\end{theorem}

Our theorem is a  generalization of the classical result \cite[Theorem 3.5]{mawhin2010critical} where the authors proved that under a quadratic growth condition on $V$, there exists a periodic solution to the problem $\ddot{u}=\nabla V(t,u)$. This result was further extended by Tian and Ge (see \cite[Theorem 2.1]{TiaGe07}) to p-Laplacian setting. They assumed that $V$ has a p-power growth.

\begin{proof}
System \eqref{eq:EL} is a system of Lagrange equations for the Lagragian function $L(t,q,p)=\Phi(p)-V(t,q)$.
Alternatively,  we can use the Lagrangian function $L(t,q,p)=\Phi(p/\Lambda)-V(t,q/\Lambda)$.
Clearly,  periodic solutions of one system correspond to periodic solutions of the other one.
The associated Hamiltonian  $\Hcal\colon [0,T]\times \R^{2n}\to \R$ is given by
\begin{equation*}
\label{main:eq:hamiltonian}
\Hcal(t,z)=
 \Phi^\star \left(\Lambda z_2\right)+V \left(t,\frac{z_1}{\Lambda}\right),
\end{equation*}
where $z=(z_1,z_2)$.

For a.e. $t\in [0,T]$, the function $\Hcal(t,\cdot)$ is
convex and $C^1$.  For every $z\in \R^{2n}$
and a.e. $t\in[0,T]$,
\begin{equation*}
\label{main:ineq:hamiltonian_lower_estimate}
\mathcal{H}(t,z)\geq
 \Phi^\star\left(\Lambda z_2\right) + \frac{1}{\Lambda}\,\inner{l(t)}{z_1}_{\R^N}
\geq\frac{1}{\Lambda}\inner{(l(t),0)}{z}_{\R^{2n}}
\end{equation*}
and
\begin{equation*}
\label{main:ineq:hamiltonian_upper_estimate}
\mathcal{H}(t,z)
\leq
\Phi^\star\left(\Lambda z_2 \right) +
\Phi\left(\Lambda z_1\right)
+\gamma(t)
=
G(\Lambda z)+ \gamma(t).
\end{equation*}

Moreover,
\begin{equation*}
    \label{main:hamiltonian_coercive}
    \int_0^T \mathcal{H}(t,z) \,dt
    =
       \Phi^\star\left(\Lambda z_2\right) T
    +
      \int_0^T  V\left(t,\frac{z_1}{\Lambda}\right) \,dt \to \infty,
    \quad z\to \infty
\end{equation*}
Hamiltonian $\Hcal$ satisfies assumptions of Theorem \ref{thm:solution-ham}. Hence, the corresponding Hamiltonian system with periodic boundary conditions has a solution $z\in \WLGTspace([0,T],\R^{2n})$. Consequently, $u=z_1/\Lambda$ is a solution of \eqref{eq:EL}. Since $\dot{u} =  \nabla \Phi^\star(z_2/\alpha)$ and $z_2\in\LGspace[\Phi^\star]([0,T],\R^N)$, then $u\in \WLGTspace[\Phi]([0,T],\R^N)$. This finishes the proof.

\end{proof}


\bibliographystyle{plain}
\bibliography{biblio}

\end{document}